\documentclass[11pt]{amsart}
\usepackage[all]{xy}
\usepackage{amsmath}
\usepackage{amsfonts,mathrsfs,bbding}
\usepackage{amssymb}
\usepackage{amscd}
\usepackage{amsthm}
\usepackage[titletoc]{appendix}
\usepackage{latexsym}
\usepackage{graphicx}
\usepackage{epstopdf}
\usepackage{amsrefs}
\usepackage{mathscinet}
\usepackage{epsfig}
\usepackage{amsthm}
\usepackage{latexsym}
\usepackage[latin1]{inputenc}
\usepackage{verbatim}
\usepackage{pb-diagram}
\usepackage[all]{xy}
\usepackage{lipsum}
\AppendGraphicsExtensions{.gif}

\theoremstyle{definition}
\newtheorem{definition}{Definition}[section]
\newtheorem{example}{Example}[section]

\theoremstyle{remark}

\theoremstyle{plain}
\newtheorem{theorem}{Theorem}[section]
\newtheorem{corollary}{Corollary}[theorem]

\def\H{\mathcal{H}}

\def\Amor{\phi}
\def\Hmor{I}
\def\limind{\mathop{\oalign{lim\cr\hidewidth$\longrightarrow$\hidewidth\cr}}} 

\newcommand{\compcent}[1]{\vcenter{\hbox{$#1\circ$}}}
\newcommand{\comp}{\mathbin{\mathchoice
{\compcent\scriptstyle}{\compcent\scriptstyle}
{\compcent\scriptscriptstyle}{\compcent\scriptscriptstyle}}} 

\title{On inductive limit spectral triples}

\author{Remus Floricel}
\address{University of Regina, Department of Mathematics, Regina, SK, Canada}
\email{Remus.Floricel@uregina.ca}
\author{Asghar Ghorbanpour} 
\address{Western University, Department of Mathematics, London, ON, Canada}
\email{aghorba@uwo.ca}
\thanks{R. Floricel was partially supported by a research grant from NSERC. A. Ghorbanpour was partially supported by a PIMS postdoctoral fellowship, held at the University of Regina.}

\date{\today}

\begin{document}

\maketitle
\begin{abstract}
Given an inductive system of spectral triples $\{(A_j,\H_j,D_j)\}_j$, we find conditions under which the triple $(\limind A_j,\limind H_j,\limind D_j)$ is a spectral triple. We also analyze and describe some classical examples of spectral triples in terms of these conditions.

\end{abstract}

\section{Introduction}
In Alain Connes' noncommutative geometry \cite{Connes1994}, the geometric information carried by a $C^*$-algebra $A$ is deciphered through conversion in spectral information, a process that is obtained with the help of a Dirac-type operator $D$ that acts on the Hilbert space $\H$ on which the given algebra is represented. The constituent objects $A$, $\H$ and $D$ of this process define the notion of spectral triple. More specifically \cite{Connes1989, Connes1994}, a spectral triple $(A,\H,D)$ consists of a unital $C^*$-algebra $A$, a unital faithful $\ast$-representation $\pi$ of $A$ on a Hilbert space $\H$, and a selfadjoint operator $D:Dom(D)\subseteq \H\to\H$ that satisfy the following conditions: \begin{enumerate}
\item [(ST1)] $D$ has compact resolvent $R_\lambda(D)=(D-\lambda I)^{-1}$, $\lambda\in\mathbb{C}\setminus\sigma (D)$; 
\item [(ST2)] the *-algebra $A^\infty$, defined as the set of all elements $a\in A$ with the property that the commutant $[D,\pi(a)]$ is densely defined and extends to a bounded operator on $\H$, is dense in $A$.
\end{enumerate}
In the literature, one can find several examples of spectral triples $(A,\H,D)$ that are built on inductive limit $C^*$-algebras  $A=\limind A_j$, the most relevant to our purpose being Connes' spectral triple on the commutative $C^*$-algebra of all continuous functions on a Cantor set  \cite[Section IV.3.$\epsilon$]{Connes1994}, and Christensen-Ivan's spectral triples on  approximately finite-dimensional ($AF$) $C^*$-algebras  \cite{Christensen-Ivan2006}. A common feature of these apparently unrelated spectral triples is that their Hilbert spaces $\H$ and Dirac operators $D$  can also be realized as inductive limits of Hilbert spaces $\H=\limind \H_j$, respectively of selfadjoint operators $D=\limind D_j$, in such a way that the constituent triples $(A_j,\H_j,D_j)$ are themselves spectral triples. Furthermore, the connecting maps that occur in all these inductive limit constructions are compatible in a natural way, making the systems of spectral triples $\{(A_j,\H_j,D_j)\}_j$ into inductive systems. We therefore infer that all these examples of spectral triples emerge as inductive limits of inductive systems of spectral triples. However, in general, the triple  $$(A=\limind A_j,\H=\limind H_j,D=\limind D_j)$$ obtained from an arbitrary inductive system  $\{(A_j,\H_j,D_j)\}_j$ of spectral triples, and referred to throughout this paper as the inductive realization of  $\{(A_j,\H_j,D_j)\}_j$, is not necessarily a spectral triple. More precisely, it may happen that the operator $D$ does not satisfy either condition (ST1), condition (ST2), or both.

It is the main goal of this work to find conditions under which the inductive realization $(A,\H, D)$ of a countable inductive system of spectral triples  $\{(A_j,\H_j,D_j)\}_j$ is still a spectral triple. The compactness of the resolvent of the operator $D$ is investigated in Theorem \ref{Diracprop}, where several equivalent conditions are discussed. A condition that implies the fulfilment of (ST2) is obtained in Corollary \ref{corfin}. The feasibility of these conditions is then tested both for Connes' spectral triple and for Christensen-Ivan's spectral triple.

At this end of the section, we mention 
that a construction of spectral triple using inductive limit techniques was explicitly obtained by Aastrup, Grimstrup and Nest in \cite{AGN} with the purpose of formulating a quantization scheme within the framework of noncommutative geometry using quantum gravity \cite{AGN1}. Their construction was later generalized by Lai in \cite{Lai2013}.  Nevertheless, the AGN spectral triple constructed in  \cite{AGN} is semifinite in the sense of \cite{CPS}, and both the framework and the themes addressed in  \cite{AGN} and \cite{Lai2013} are different from ours.

\section{Background: definitions and examples} As a categorical process, the concept of inductive limit of spectral triples requires a functional definition of the notion of morphism of spectral triples. Depending on the intended purpose, several definitions have been proposed in the literature, each of them having its own advantages and disadvantages (see \cite{Bertozzini2012} for a thorough discussion). The use, in this paper, of the following definition of morphism of spectral triples is justified by its occurrence in all concrete examples encountered.
\begin{definition}\label{morph1}
 A morphism between two spectral triples $(A_1,\H_1,D_1)$ and $(A_2,\H_2,D_2)$ is a pair $(\phi,I)$ consisting of a unital $\ast$-homomorphism $\phi:A_1\to A_2$ and a bounded linear operator $I:\H_1\to\H_2$ satisfying the following conditions:\begin{enumerate}
\item $\phi(A_1^\infty)\subseteq A_2^\infty$, where $A_1^\infty$ and $A_2^\infty$ are as in (ST2);
\item $I\pi_1(a)=\pi_2(\phi(a))I$, for every $a\in A_1$;
\item $I(Dom(D_1))\subseteq Dom(D_2)$ and $I D_1=  D_2 I$. 
\end{enumerate}
 A morphism $(\phi,I)$ is said to be isometric if $\phi$ is injective and $I$ is an isometry.
\end{definition}
Using this definition of morphism, one can naturally introduce the notion of inductive system of spectral triples, as follows.
\begin{definition}Let $(J,\leq)$ be a directed index set and $\{(A_j,\H_j,D_j)\}_{j\in J}$ be a family of  spectral triples. 
Suppose that for every $j,\,k\in J$, $j\leq k$, an isometric morphism $(\phi_{j,k},I_{j,k})$ from $(A_j,\H_j,D_j)$ to $(A_k,\H_k,D_k)$ is given such that $\phi_{k,l} \phi_{j,k}=\phi_{j,l}$ and $I_{k,l}I_{j,k}=I_{j,l}$ for all $j,\,k,\,l\in J$, $j\leq k\leq l$. 
The resulting system  $\{(A_j,\H_j,D_j), (\phi_{j,k},I_{j,k}) \}_J$ is called an inductive system of spectral triples. 
\end{definition}
It follows from the previous definition that, given an inductive system of spectral triples $\{(A_j,\H_j,D_j), (\phi_{j,k},I_{j,k}) \}_J$, the constituent systems $\{(\H_j, I_{j,k})\}_J$ and $\{(A_j,\phi_{j,k}) \}_J$ are inductive systems of Hilbert spaces, respectively of $C^*$-algebras.
Let 
\begin{equation}\label{inuctivepieces}
\H=\limind_{j\in J}\H_j\;\;\mbox{and}\;\;A=\limind_{j\in J}A_j
\end{equation} 
be their inductive limits with associated connecting isomeric operators $I_j:\H_j\to \H$, respectively *-monomorphisms $\phi_j:A_j\to A$,  $j\in J$. We also consider the inductive limit $\pi=\limind_{j\in J}\pi_j$ of the family of representations $\{\pi_j\}_{j\in J}$ associated with the family of spectral triples  $\{(A_j,\H_j,D_j)\}_{j\in J}$. Therefore $\pi$ is the unique representation of the $C^*$-algebra $A$ on $\H$ such that \begin{eqnarray}\label{repo}\pi(\phi_j(a))I_j=I_j\pi_j(a),\;a\in A_j,\,j\in J.\end{eqnarray}
Next, we briefly outline how to construct the inductive limit of the family of operators $\{D_j\}_{j\in J}$. For this purpose, we consider the dense domain $\mathscr{D}$ of $\H$, \begin{eqnarray}\label{dinfty}\mathscr{D}=\bigcup_{j\in J}I_j(Dom(D_j)).\end{eqnarray} For every vector $\xi\in \mathscr{D}$ of the form $\xi=I_j\xi_j$, where $\xi_j\in Dom(D_j)$, define  \begin{eqnarray}\label{ops}D\xi=\Hmor_jD_j\xi_j.\end{eqnarray} 
It follows that  $D$ is a densely defined symmetric operator.
Moreover, since the operators $D_j$ are self-adjoint, we have that $Range(D_j\pm i)=\H_j$, for every $j\in J$. 
Consequently,  $Range(D\pm i$) is dense in $\H$, and thus $D$ is essentially selfadjoint.  

We shall use the same letter $D$, or the symbol $D=\limind D_j$, to denote the closure of this essentially selfadjoint operator, and call it the inductive limit of the family of operators $\{D_j\}_{j\in J}$. Therefore $D$ is a selfadjoint operator of which domain $Dom (D)$ contains $\mathscr{D}$, and hence $$I_j(Dom(D_j))\subseteq Dom(D),$$ for every $j\in J$. In particular, if $P_j$ is the orthogonal projection of $\H$ onto $I_j(\H_j)\simeq \H_j$, then $P_j (Dom(D))\subseteq Dom(D)$, and the operators $P_j$ and $D$ commute with each other.

Putting together all the elements defined above, we come to the next definition.
\begin{definition}
The triple $(A,\H,D)$ is called the inductive realization of the inductive system $\{(A_j,\H_j,D_j), (\phi_{j,k},I_{j,k}) \}_{J}$.
\end{definition}

The following examples illustrate the concepts introduced above.

\begin{example}\label{Co}(cf. \cite[Section IV.3.$\epsilon$]{Connes1994}.) Connes' spectral triple associated with a Cantor set can be described alternatively as the inductive realization of an inductive system of finite dimensional spectral triples, as follows.

Consider a Cantor set $\Lambda \subset \mathbb{R}$, i.e., a totally disconnected compact Hausdorff space with no isolated points. Set $x_{0,+}=\min\Lambda$, $x_{0,-}=\max\Lambda$, and $\ell_0=x_{0,-}-x_{0,+}$. Let $\{(x_{n,-},x_{n,+})\}_{n\in\mathbb{N}}$ be a sequence of disjoint open intervals of lengths $\ell_n=x_{n,+}-x_{n,-}$ decreasing to zero such that 
$$\Lambda=[x_{0,+},x_{0,-}]\backslash \bigcup_{n=1}^\infty (x_{n,-},x_{n,+}).$$ 

Let $j\in\mathbb{N}$ be a fixed integer. Consider the subset $\Lambda_j=\{x_{n,+}\}_{0\leq n\leq  j}$ of $\Lambda$, and the continuous function $\theta_j:\Lambda\to\Lambda_j$, $$\theta_{j}(x)=\max \{x_{n,+}\,|\, x_{n,+}\leq x,\,0\leq n\leq j \}.$$ For any $k\in\mathbb{N}$, $k\geq j$, we denote with $\theta_{j,k}$ the restriction of $\theta_{j}$ to $\Lambda_k$. Then $\theta_{j,k}\theta_k=\theta_j$, and  $\theta_{j,k}\theta_{k,l}=\theta_{j,l}$ for every $l\geq k\geq j$, as one can readily see.

Let $A=C(\Lambda)$ be the $C^*$-algebra of all continuous functions on $\Lambda$, $A_j=C(\Lambda_j)$, and  $\varphi_{j}:A_j\to A$ be the *-homomorphism induced by $\theta_{j}$, i.e.  $\phi_{j}(f)=f\comp \theta_{j}$. For $k\in\mathbb{N}$, $k\geq j$, let also $\phi_{j,k}:A_j\to A_k$ be  the  the *-homomorphism induced by $\theta_{j,k}$. Then $\phi_{k}\phi_{j,k}=\phi_{j}$ and $\phi_{j,k}\phi_{k,l}=\phi_{j,l}$, for every $l\geq k\geq j$. Moreover, the *-algebra $\bigcup_{j\in\mathbb{N}}\Amor_{j}(A_j)$ is dense in $A$, by Stone-Weierstrass' theorem. Consequently $\{(A_j,\phi_{j,k}) \}_{\mathbb{N}}$ is an inductive system of $C^*$-algebras and $A=\limind A_j$. \\
\indent Next, we consider the sets $E_j=\{x_{n,\pm}\}_{0\leq n \leq j }$ and  $E=\{x_{n,\pm}\}_{n\in\mathbb{N}}$, and  the associated Hilbert spaces $\mathcal{H}_j=\ell^2\big(E_j)$ and $\mathcal{H}=\ell^2\big(E)$. Then $\mathcal{H}=\limind \mathcal{H}_j$, where the connecting isometries $I_{j,k}:\H_j\to H_k$ are given by inclusion, i.e.,  $(I_{j,k}\xi)(x_{n,\pm})=\xi(x_{n,\pm})$ if $n\leq j$, and $(I_{j,k}\xi)(x_{n,\pm})=0$ if $j<n\leq k$, for every $\xi\in \H_j$. 
We also consider the representation $\pi_j$ of the $C^*$-algebra $A_j$ on the Hilbert space $\H_j$, defined as $$(\pi_j(f)(\xi))(x_{n,\pm})=f(\theta_j(x_{n,\pm}))\xi(x_{n,\pm}),$$ for all $f\in A_j$, $\xi\in\mathcal{H}_j$, and $0\leq n\leq j,$ and the operator $D_j:\mathcal{H}_j\to\mathcal{H}_j,$  \begin{eqnarray}\label{dj}D_j(\xi)(x_{n,\pm})=\frac{1}{\ell_n} \xi(x_{n,\mp}),\end{eqnarray} defined for all $\xi\in\mathcal{H}_j$, and $0\leq n\leq j.$ Note that if $\pi:A\to \mathcal{B}(\H)$ is the representation $$ (\pi(f)(\xi))(x)=f(x)\xi(x),$$  defined for all $f\in A$, $\xi\in H$ and $x\in E$, then $\pi=\limind \pi_j$, and $\|D_j\|=\frac{1}{\ell_j}$. 

Putting all the elements defined above together, we notice that the system $\{(A_j,\H_j,D_j), (\phi_{j,k},I_{j,k}) \}_{\mathbb{N}}$ is an inductive system of finite dimensional spectral triples. The inductive realization of this inductive system is precisely Connes' spectral triple $(C(\Lambda), \ell^2(E), D)$ constructed in \cite[Section IV.3.$\epsilon$]{Connes1994} (see also \cite{Connes-Marcolli}). The Dirac operator $D$ acts as in (\ref{dj}), i.e. $D(\xi)(x_{n,\pm})=\frac{1}{\ell_n} \xi(x_{n,\mp})$, for every function $\xi:E\to \mathbb{C}$ of finite support and $n\in \mathbb{N}$, and the algebra $A^\infty$ is simply the algebra of locally constant functions on $\Lambda$.

At this end, we also note that the spectral triple $(C(\Lambda), \ell^2(E), D)$ is even with $\mathbb{Z}/2\mathbb{Z}$-grading $\gamma : \ell^2(E)\to  \ell^2(E)$, $\gamma(\xi)(x_{n,\pm})=\xi(x_{n,\mp}),$ and the operator $\gamma$ can also be realized as the inductive limit of a sequence of  $\mathbb{Z}/2\mathbb{Z}$-grading operators $\gamma_j : \H_j\to  \H_j$, defined similarly. To keep things simple, in this article we will not discuss in detail inductive limits of inductive systems of even spectral triples. It is clear, however, that all the concepts defined above can be adjusted with ease to address this situation.
\end{example}

\begin{example}\label{ChIv}(cf. \cite{Christensen-Ivan2006}). Christensen and Ivan have constructed in \cite{Christensen-Ivan2006} a spectral triple for unital $AF$-algebras which, when considered on commutative $AF$-algebras, differs in several ways from Connes' spectral triple discussed in the previous example. The Christensen-Ivan spectral triple was constructed as the result of an effective utilization of the inductive structure of an $AF$-algebra, and can be easily described as the inductive realization of an inductive system of finite dimensional spectral triples, as follows. 

Let $A$ be a unital $AF$-algebra, $\{A_j\}_{j\geq 0}$ be an increasing sequence $$A_0=\mathbb{C}\cdot 1\subseteq A_1 \subseteq  \cdots\subseteq A$$ of  
finite dimensional $C^\ast$-algebras of which union is dense in $A$, and $\{\alpha_j\}_{j\in\mathbb{N}}$ be a sequence of non-zero real numbers. Consider a faithful state $\tau$ of $A$. Let $(\pi, \H)$ be the the associated GNS representation with cyclic (and separating) vector $\xi\in \H$, and $\eta:A\to \H$ be the mapping  $\eta(a)=\pi(a)\xi$, $a\in A$. For each positive integer $j$, consider the finite dimensional subspace $\H_j=\eta(A_j)$ of $\H$, the representation $\pi_j$  of $A_j$ on $\H_j$, $\pi_j(a)\eta(b)=\eta(ab)$, for all $a,\,b\in A_j$, and the orthogonal projection $P_j$ of $\H$ onto $\H_j$. Let also $D_j:\H\to \H$ be the operator  $$D_j=\sum_{i=1}^j\alpha_i(P_i-P_{i-1}).$$ It is clear that each Hilbert space $\H_k$ is invariant under $D_j$, and $D_j\restriction_{\H_j}=D_k\restriction{\H_j}$ for all $j\leq k$. We deduce that the system $\{(A_j,\H_j,D_j\restriction _{\H_j}), (\phi_{j,k},I_{j,k}) \}_{\mathbb{N}}$ is an inductive system of finite dimensional spectral triples, where $\phi_{j,k}$ and $I_{j,k}$ are the inclusion maps. The inductive realization $(A,\H,D)$ of this system gives the Christensen-Ivan spectral triple. \end{example}

\section{Inductive realizations as spectral triples}
Throughout this section, we consider a countable inductive system of spectral triples $\{(A_j,\H_j,D_j), (\phi_{j,k},I_{j,k}) \}_{\mathbb{N}}$ and its inductive realization $(A,\H,D)$. As in the previous section, we also consider the increasing sequence $\{P_j\}_{j\in\mathbb{N}}$ of orthogonal projections $P_j$ of $\H$ onto $I_j(\H_j)\simeq \H_j$, where $I_j:\H_j\to \H$ are the connecting isometries of the inductive system of Hilbert spaces $\{(\H_j,I_{j,k}) \}_{j\leq k}$ satisfying $I_kI_{j,k}=I_j$, for all $k\geq j$, and $\bigcup_{i\in\mathbb{N}}I_j(\H_j)$ is dense in $\H$.

In the first part of this section, we introduce and analyze conditions under which the operator $D$ has compact resolvent. Our strategy for achieving this goal starts from the observation that the operator $D$ can be realized as the limit of the sequence $\{I_jD_jI_j^*\}_{j\in\mathbb{N}}$ in the strong resolvent sense, i.e., $${\mbox{\tiny{SOT}}}-\lim_{j\to\infty}R_\lambda (I_jD_jI_j^*)=R_\lambda(D)$$ for $\lambda\in\mathbb{C}\setminus\mathbb{R}$, a property that has also been noted in \cite{Lai2013}.  In particular, it follows from \cite[Theorem VIII.20(b)]{Reed-Simon1980} that the sequence $\{f(I_jD_jI_j^*)\}_{j\in\mathbb{N}}$ converges strongly to $f(D)$, for every bounded continuous function $f$ on $\mathbb{R}$.

The convergence of the sequence $\{I_jD_jI_j^*\}_{j\in\mathbb{N}}$ to $D$ in the strong resolvent sense can be easily seen. Indeed, since  $P_j\uparrow 1_\H$ strongly, it follows that for every $\lambda\in\mathbb{C}\setminus\mathbb{R}$, we have 
\begin{equation}\label{sotconver}
I_jR_\lambda(D_j)I_j^*=P_jR_\lambda(D)P_j\overset{\rm \tiny{SOT}}{\rightarrow} R_\lambda(D). 
\end{equation}  
Moreover, because $R_\lambda(I_jD_jI_j^*)=I_jR_\lambda(D_j)I_j^*+\lambda P_j^\perp,$
for every $j\in\mathbb{N}$, we deduce that  the sequence $\{I_jD_jI_j^*\}_{j\in\mathbb{N}}$ converges to $D$ in the strong resolvent sense. In this regard, we also notice that the sequence $\{I_jD_jI_j^*\}_{j\in\mathbb{N}}$ converges to $D$ in the norm resolvent sense if and only if $I_j$ is a unitary operator, for some $j\in \mathbb{N}$. Consequently, this kind of convergence is far too strong for the level of generality we aim to maintain in this work.

To ensure that condition (ST1) is met, we therefore need to approximate the operator $D$ with the given sequence $\{D_j\}_{j\in\mathbb{N}}$ through a different kind of convergence, which should be stronger than convergence in the strong resolvent sense, but weaker than convergence in the norm resolvent sense. Equation (\ref{sotconver}) offers a clear indication of the type of convergence required by replacing the strong operator topology with the uniform topology, as discussed in the theorem below.
\begin{theorem} \label{Diracprop}
Let $\{(A_j,\H_j,D_j), (\phi_{j,k},I_{j,k}) \}_{\mathbb{N}}$ be an inductive system of spectral triples with inductive realization $(A,\H,D)$. The following conditions are equivalent.
\begin{enumerate}
\item[{(\em i})] $D$ has compact resolvent;
 \item[({\em ii})] the sequence $\{I_jR_\lambda(D_j)I_j^*\}_{j\in\mathbb{N}}$ converges uniformly to $R_\lambda(D)$, for every $\lambda\in\mathbb{C}\setminus\mathbb{R}$;
  \item[({\em iii})] the sequence $\{I_jR_\lambda(D_j)I_j^*\}_{j\in\mathbb{N}}$ converges uniformly to $R_\lambda(D)$, for some $\lambda\in\mathbb{C}\setminus\mathbb{R}$;
  \item[({\em iv})] the sequence $\{I_jf(D_j)I_j^*\}_{j\in\mathbb{N}}$ converges uniformly to $f(D)$, for every continuous function $f$ on $\mathbb{R}$ vanishing at infinity.
\end{enumerate}
\end{theorem}
\begin{proof}
$(i)\Rightarrow (ii).$ Let $\lambda\in\mathbb{C}\setminus\mathbb{R}$ be fixed. Since the resolvent $R_\lambda(D)$ is a compact operator, we have
$R_\lambda(D)=\sum_n (\lambda_n-\lambda)^{-1} Q_{\lambda_n},$
where $Q_{\lambda_n}$ is the eigenprojection for the eigenvalue $\lambda_n$ of $D$, and $\lim_n|\lambda_n-\lambda|^{-1}=0$.
Each projection $P_j$ commutes with $D$ so it also commutes with any of the eigenprojections  $Q_{\lambda_n}$. One can therefore consider the (possibly zero) projections $Q_{j,\lambda_n}:=P_j Q_{\lambda_n}\leq Q_{\lambda_n}$. Since  $P_j\uparrow 1_\H$ strongly and  $Q_{\lambda_n}$ have finite ranks, we deduce that $Q_{j,\lambda_n}=Q_{\lambda_n}$ for $j$ large enough. Moreover, since $DI_j=I_jD_j$, we obtain
\begin{eqnarray}\label{norm}
\|I_j R_\lambda(D_j)I_j{}^*-R_\lambda(D)\|
&=&\sup\left.\left\{|\lambda_n-\lambda|^{-1}\, \right| {Q_{j,\lambda_n}\neq Q_{\lambda_n}}\right\}. 
\end{eqnarray}   
For any $\epsilon>0$, there exists $m$ such that  $|\lambda_n-\lambda|^{-1}<\epsilon$, for every $n>m$.
On the other hand, for $j$ large enough,  we have $Q_{j,\lambda_k}=Q_{\lambda_k}$ for all $k\leq m $. In particular, it follows from \eqref{norm} that  $\|I_j R_\lambda(D_j)I_j{}^*-R_\lambda(D)\|<\epsilon$, for $j$ large enough. \\
$(ii)\Rightarrow (iii) $ is obvious.\\ $(iii)\Rightarrow (ii)$ follows from the resolvent identity.\\
$(ii)\Rightarrow (iv).$ Let  $f$ be a continuous function vanishing at infinity, $\epsilon>0$, and $\lambda\in\mathbb{C}\setminus\mathbb{R}$. Using the Stone-Weierstrass theorem, one can find a polynomial $p$ in two variables, with no constant term, such that $\|f-p_\lambda\|<\frac{\epsilon}{3}$, where $p_\lambda(x)=p(1/(x-\lambda),1/(x-\overline{\lambda}))$, for every real number $x$. Using functional calculus, we obtain$$\| f(D) -p_\lambda(D)\|< \frac{\epsilon}{3}\;\mbox{and}\;  \| f(D_j) -p_\lambda(D_j)\|<\frac{\epsilon}{3},$$
for all $j\in\mathbb{N}$.  
In particular, 
 $\|I_j f(D_j)I_j{}^*-I_j p_\lambda(D_j)I_j{}^*\|<\frac{\epsilon}{3},$ for every $j$.
Note that  \begin{eqnarray*}I_jp_\lambda(D_j)I_j^*&=&I_jp\big(R_\lambda(D_j),R_{\overline{\lambda}}(D_j)\big)\Hmor_j{}^*\\&=&p\big(I_jR_\lambda(D_j)I_j{}^*,I_jR_{\overline{\lambda}}(D_j)I_j{}^*\big),\end{eqnarray*}
 because $I_j\big(R_\lambda (D_j)\big)^{m}\big(R_{\overline{\lambda}}(D_j)\big)^{n}I_j{}^*=\Big(I_j R_\lambda(D_j)I_j{}^*\Big)^m\Big(I_jR_{\overline{\lambda}}(D_j)I_j{}^*\Big)^n$ for all $m,\,n\neq 0$. Moreover, because the sequences $\{I_jR_\lambda(D_j)I_j{}^*\}_{j\in\mathbb{N}}$ and $\{I_jR_ {\overline{\lambda}}(D_j)I_j{}^*\}_{j\in\mathbb{N}}$ converge uniformly to the resolvents $R_\lambda(D)$ and $R_{\overline{\lambda}}(D)$ respectively, one can choose $j$ large enough so that 
$$\| I_jp_\lambda(D_j)I_j{}^*-p_\lambda(D)\|<\frac{\epsilon}{3}.$$ Consequently $\|f(D)-\Hmor_jf(D_j)\Hmor_j^* \|<\varepsilon$, for $j$ large enough.\\
$(iv)\Rightarrow (i).$ Let $\lambda\in\mathbb{C}\setminus\mathbb{R}$ and $f(x)=1/(x-\lambda)$, for every real number $x$. Then the sequence of compact operators $\{I_jf(D_j)I_j^*\}_{j\in\mathbb{N}}$ converges uniformly to $R_\lambda(D)$, so the resolvent  $R_\lambda(D)$ must be a compact operator as well.\end{proof}

In the following, we will focus on investigating the validity of condition (ST2) in the case of an inductive realization.
Our strategy consists of comparing the algebra $A_\infty:=\bigcup_{j\in\mathbb{N}} \phi_j(A_j^\infty)$ to the algebra $A^\infty$, defined in (ST2), two algebras that may be unrelated.
For this purpose, we first notice that for any $a$ in some $A_j^\infty$, the operator $\pi(\phi_j(a))$ preserves $I_j(Dom(D_j))$ because $\pi_j(a)$ preserves $Dom(D_j)$.
Therefore, $\pi(\phi_j(a))$ preserves the space $\mathscr{D}$, defined in (\ref{dinfty}), and thus the commutator $[D,\pi(\phi_j(a))]$ is a densely defined operator.

Secondly, we notice that if $a\in A_j^\infty$, for some $j\in\mathbb{N}$, then for every $k\geq j$ one has
\begin{eqnarray*}
[D,\pi(\phi_j(a))] I_k 
&=& DI_k\pi_k(\phi_{j,k}(a))-\pi(\phi_k\comp \phi_{j,k}(a))I_k D_k\\
&=&  I_k D_k  \pi_k(\phi_{j,k}(a))-I_k \pi_k(\phi_{j,k}(a)) D_k\\
&=&  I_k [D_k, \pi_k(\phi_{j,k}(a)].
\end{eqnarray*}
This, on one hand, shows that   
\begin{equation}\label{firstcommutequ}
I_k^*[D,\pi(\phi_j(a))] I_k=[D_k, \pi_k(\phi_{j,k}(a)]\in \mathcal{B}(\H_k).
\end{equation}
On the other hand, it implies that for every selfadjoint element $a\in A_j^\infty$ the projection  $P_k=I_kI_k{}^\ast$, for $k\geq j$, commutes with the commutator $[D,\pi(\phi_j(a))]$;
\begin{equation}
[D,\pi(\phi_j(a))]P_k=P_k[D,\pi(\phi_j(a))], \quad a=a^*\in A_j^\infty.
\end{equation} 
We also infer from (\ref{firstcommutequ}) that $I_{\ell,k}^*[D_k, \pi_k(\phi_{j,k}(a)]I_{\ell,k}=[D_\ell, \pi_\ell(\phi_{j,\ell}(a)],$ for every $j\leq \ell\leq k$, and therefore the sequence $\{\|[D_k,\pi_k(\phi_{j,k}(a))]\|\}_{j\geq k}$ is increasing. The convergence of this sequence, for every  $a\in A_j^\infty$, provides a necessary and sufficient condition that ensures the inclusion  $\phi_j(A_j^\infty)\subset A^\infty$, as shown in the following theorem.

\begin{theorem}\label{boundedcomutant}
Let $\{(A_j,\H_j,D_j), (\phi_{j,k},I_{j,k}) \}_\mathbb{N}$ be an inductive system of spectral triples. Let $j\in\mathbb{N}$ and $a\in A_j^\infty$. Then the operator $[D, \pi(\phi_{j}(a))]$ is bounded if and only if 
 the family of operators $\{[D_k,\pi_k(\phi_{j,k}(a))]\}_{k\geq j}$ is uniformly bounded.
\end{theorem}
\begin{proof}$(\Rightarrow)$ Using \eqref{firstcommutequ}, we have $\|[D_k,\pi_k(\phi_{j,k}(a))]\|\leq \|[D,\pi(\Amor_j(a))]\|,$ for all $k\geq j, $
which proves one direction.

$(\Leftarrow)$ Without loss of generality, we can assume that $a\in A_j^\infty$ is selfadjoint. For every $k\geq j$, the operator $[D,\pi(\phi_j(a))]P_k=P_k[D,\pi(\Amor_j(a))]P_k$ is bounded by  \eqref{firstcommutequ}. Furthermore, since $$\|(P_{k+1}-P_k) [D,\pi(\phi_j(a))] (P_{k+1}-P_k)\|\leq 
 \|[D_k,\pi_k(\phi_{j,k}(a))]\|,$$ we obtain that the family of bounded operators $$\{(P_{k+1}-P_k) [D,\pi(\phi_j(a))] (P_{k+1}-P_k)\}_{k\geq j}$$ is uniformly bounded. In particular, since $(P_{l+1}-P_l)(P_{k+1}-P_k)=0$ for every $l>k$, the operator $\sum_{k\geq j}(P_{k+1}-P_k) [D,\pi(\phi_j(a))] (P_{k+1}-P_k)$ is well-defined and bounded. Moreover, $$\sum_{k\geq j}(P_{k+1}-P_k) [D,\pi(\phi_j(a))] (P_{k+1}-P_k)=[D,\pi(\phi_j(a)] -[D,\pi(\phi_j(a))]P_j$$ on the domain of $D$, as one can readily see. 
Therefore $[D,\pi(\phi_j(a)]$ is bounded, which concludes the proof.

\end{proof}
The following corollary follows directly from the previous proposition, and from the fact that the algebra $A_\infty$ is dense in $A$.
\begin{corollary}\label{corfin}
Suppose that  the family of operators $\{[D_k,\pi_k(\phi_{j,k}(a))]\}_{k\geq j}$ is uniformly bounded, for every  $a\in A_j^\infty$ and $j\in\mathbb{N}$. Then the operator $D$ satisfies condition (ST2).
\end{corollary}

 At this the end of the paper, we test the effectiveness of the results obtained in Theorem \ref{Diracprop} and Corollary \ref{corfin}  on the two examples discussed in section two. Of course, in both cases the properties (ST1) and (ST2) can be directly verified. We will use the same notation as the one in Example \ref{Co} and Example \ref{ChIv}.
 \subsection*{Example \ref{Co} revisited} {\em Condition (ST1)}. For every $\lambda\in\mathbb{C}\setminus\mathbb{R}$, one has $$\|I_jR_\lambda(D_j)I_j^\ast-R_\lambda(D)\|=\sup_{k> j}\frac{1}{|\ell_k^{-1}-\lambda|} \leq \ell_j,$$ and since $\ell_j\to 0$, we obtain from Theorem \ref{Diracprop} that $D$ has compact resolvent.\\
 {\em Condition (ST2)} is automatically satisfied, because $A_\infty\subset A^\infty$, by construction.
 \subsection*{Example \ref{ChIv} revisited}  {\em Condition (ST1)}. We first notice that the operator $D=\limind D_j$ is bounded iff the sequence $\{\alpha_j\}_{j\in\mathbb{N}}$  is bounded, because $\|D_j\|=\max\{|\alpha_i|\,|\,i\leq j\}$, for every $j\in\mathbb{N}$.  Therefore, in this case, the operator $D$ can not have compact resolvent, unless the $C^*$-algebra $A$ is finite dimensional. If the sequence  $\{\alpha_j\}_{j\in\mathbb{N}}$ is unbounded, then using  Theorem \ref{Diracprop} we obtain that the operator $D$ has compact resolvent if and only if$$\|I_jR_\lambda(D_j)I_j^\ast-R_\lambda(D)\|=\sup_{k> j}\frac{1}{|\alpha_k-\lambda|} \to 0,$$ for $\lambda\in\mathbb{C}\setminus\mathbb{R}$, 
 which is equivalent to $|\alpha_j|\to \infty $ as $j\to \infty$.\\
 {\em Condition (ST2)} is always satisfied. Indeed, if $j\in\mathbb{N}$ and $a\in A_j^\infty=A_j$, then $[P_k,\pi(a)]=0$, for every $k\geq j$, and thus $[D_k,\pi(a)]=[D_j,\pi(a)]$. Since $[D_k,\pi(a)]\restriction_{\H_k}=[D_k,\pi_k(a)]\restriction_{\H_k}$, it follows that $$\|[D_k,\pi_k(a)]\restriction_{\H_k}\|\leq \|[D_j,\pi(a)] \|,$$ for every $k\geq j$, and the conclusion follows from Corollary \ref{corfin}.\\\\
We close this work by pointing out that the inductive realizations of inductive systems of {\em finite dimensional} spectral triples have great potential for classification and direct description as spectral triples; this is also due to the fact that the class of finite dimensional spectral triples is perfectly classifiable in terms of the Krajewski diagrams \cite{Krajewski1998} (see also \cite{vanSuijlekom2015}). We intend to discuss these issues in a future work.


\end{document}